\documentclass[11pt]{amsart}
\usepackage{fullpage}
\usepackage{amsfonts}
\usepackage{amsmath}
\usepackage{amssymb}
\usepackage{amsthm}
\usepackage{verbatim}

\newtheorem{theorem}{Theorem}

\newtheorem{cor}[theorem]{Corollary}

\newtheorem{conj}[theorem]{Conjecture}
\newtheorem{lemma}[theorem]{Lemma}
\newtheorem{prop}[theorem]{Proposition}

\newcommand{\norm}[1]{\left|\left|#1\right|\right|}
\newcommand{\sph}{\mathbb{S}^{n-1}}

\newcommand{\rn}{\mathbb{R}^{n}}

\newcommand{\p}{P_{n,2d}}
\newcommand{\h}{H_{n,2d}}
\newcommand{\rmm}{\mathbb{R}^{m}}
\newcommand{\hd}{H_{n,d}}
\newcommand{\s}{S_{n,d}}

\newcommand{\sq}{\Sigma_{n,2d}}

\newcommand{\st}{\hspace{2mm} \mid \hspace{2mm}}

\newcommand{\ql}{Q_{\ell}}
\newcommand{\wl}{W_{\ell}}
\newcommand{\bl}{B_{\ell}}
\newcommand{\foreq}{\hspace{5mm}\text{for}\hspace{3mm}}
\def\vc{\mathcal{V}_{\mathbb{C}}}

\def\co{\mathbb{C}}
\def\rr{\mathbb{R}}
\numberwithin{theorem}{section}
\numberwithin{equation}{section}
\begin{document}
\title{Nonnegative Polynomials and Sums of Squares}
\author{Grigoriy Blekherman}
\begin{abstract}
In the smallest cases where there exist nonnegative polynomials that are not sums of squares we present a complete explanation of this distinction. The fundamental reason that the cone of sums of squares is strictly contained in the cone of nonnegative polynomials is that polynomials of degree $d$ satisfy
certain linear relations, known as the Cayley-Bacharach relations, which are not satisfied by polynomials of full degree $2d$. For any nonnegative polynomial
that is not a sum of squares we can write down a linear inequality coming from a Cayley-Bacharach relation that certifies this fact. We also characterize
strictly positive sums of squares that lie on the boundary of the cone of sums of squares and extreme rays of the cone dual to the cone of sums of squares
\end{abstract}

\maketitle
\section{Introduction}
A real polynomial in $n$ variables is called nonnegative if it is greater than or equal to $0$ at all points in $\rn$. It is a central question in real algebraic geometry, whether a non-negative polynomial can be written in a way that makes its nonnegativity apparent, i.e. as a sum of squares of polynomials (or more general objects). Algorithms to obtain such representations, when they are known, have many applications in polynomial optimization \cite{JiawangMarkus},\cite{Lasserre},\cite{pablo}.

The investigation of the relation between nonnegativity and sums of squares began in the seminal paper of Hilbert from 1888. Hilbert showed that every nonnegative polynomial is a sum of squares of polynomials only in the following 3 cases: univariate polynomials, quadratic polynomials and bivariate polynomials of degree 4. In all other cases Hilbert showed existence of nonnegative polynomials that are not sums of squares. Hilbert's proof used the fact that polynomials of degree $d$ satisfy linear relations, known as the Cayley-Bacharach relations, which are not satisfied by polynomials of full degree $2d$ \cite{Rez1},\cite{Rez2}.

Hilbert then showed that every bivariate nonnegative polynomial is a sum of squares of rational functions and Hilbert's 17th problem asked whether this is true in general. In 1920'a Artin and  Schreier solved Hilbert's 17th problem in the affirmative. However, there is no known algorithm to obtain this representation. In particular we may need to use numerators and denominators of very large degree, thus representing a simple object (the polynomial) as a sum of squares of significantly more complex objects \cite{BCR}.

It should be noted that Hilbert did not provide an explicit nonnegative polynomial that is not a sum of squares of polynomials, he only proved its existence. The first explicit example appeared only eighty years later and is due to Motzkin. Since then many explicit examples of nonnegative polynomials that are not sums of squares have appeared \cite{Rez1}. For some low dimensional, symmetric families there are also descriptions of the exact differences between nonnegative polynomials and sums of squares  \cite{CLR}. However even in the smallest cases where nonnegative polynomials are different from sums of squares, 3 variables degree 4 and 2 variables degree 6, we have not had a complete understanding of what makes nonnegative polynomials different from sums of squares.

We show that, in these cases, all linear inequalities that separate nonnegative polynomials from sums of squares come from the Cayley-Bacharach relations. The Cayley-Bacharach relations were already used by Hilbert in the original proof of existence of nonnegative polynomials that are not sums of squares. We show  that, in fact, these relations are the fundamental reason underlying the existence of any such polynomial, and we provide explicit structure for the linear inequalities separating nonnegative polynomials from sums of squares. The algebra and geometry involved in these two cases is quite similar and we give a complete unified geometric description of the differences between nonnegative polynomials and sums of squares.

\subsection{Main Results} By analogy with quadratic forms we will refer to nonnegative polynomials as \textit{positive semidefinite} or \textit{psd} for short and sums of squares will be called \textit{sos}. Any psd polynomial can be made homogeneous by adding an extra variable and it will remain nonnegative. The same holds for sums of squares. We will therefore work with homogeneous polynomials (forms).

Our goal is to investigate the cases of forms in 3 variables of degree 6, known as ternary sextics, and forms in 4 variables of degree 4, known as quaternary quartics. We will denote these cases as $(3,6)$ and $(4,4)$ respectively.

Let $\hd$ be the vector space of real forms in $n$ variables of degree $d$. Nonnegative forms and sums of squares both form full dimensional closed convex cones in $\h$, which we call $\p$ and $\sq$ respectively:

$$\p=\left\{p \in \h \hspace{2mm} \mid \hspace{2mm} p(x)\geq 0 \hspace{2mm} \text{for all} \hspace{2mm} x \in \rn\right\},$$

\noindent and
$$\sq=\left\{p \in \h \hspace{2mm} \big| \hspace{2mm} p(x)=\sum q_i^2  \hspace{2mm} \text{for some} \hspace{2mm} q_i \in \hd \right\}.$$

It is clear that $\sq \subseteq \p$ and by Hilbert's theorem this inclusion is actually strict in the cases $(3,6)$ and $(4,4)$.

The defining linear inequalities for the psd cone $\p$ are easy to describe, they are given by $$f(v) \geq 0 \hspace{7mm} \text{for all} \hspace{5mm} v \in \rn.$$
By homogeneity of forms it suffices to only consider points $v$ in the unit sphere $\sph$. We remark that with this characterization and an appropriate choice of the inner product it is not hard to show that the dual cone to $\p$ is the conic hull of the real Veronese variety of degree $2d$ and thus the dual cone of $\p$ is essentially equivalent to the Veronese Orbitope \cite{Raman}.

The above inequalities are clearly satisfied by all sums of squares but when the sos cone is strictly smaller, it must satisfy additional linear inequalities. We prove the following characterization for $(3,6)$:

\begin{theorem}\label{THEOREM MAIN (3,6)}
Suppose that $p \in P_{3,6}$ and $p$ is not sos. Then there exist two real cubics $q_1,q_2 \in H_{3,3}$ intersecting in 9 (possibly complex) projective points $\gamma_1,\ldots,\gamma_9$ such that the values of $p$ on $\gamma_i$ certify that $p$ is not a sum of squares. More precisely, let $z_1,\ldots,z_9$ be affine
representatives of $\gamma_i$. Then there exists a real linear functional $\ell: H_{3,6} \rightarrow \mathbb{R}$ given by
$$\ell(f)=\sum \mu_if(z_i),$$
for some $\mu_i \in \mathbb{C}$ such that $\ell(r) \geq 0$ for all $r \in \Sigma_{3,6}$ and $\ell(p)<0$. Furthermore at most 2 of the points $\gamma_i$ are complex.
\end{theorem}

We also prove a similar theorem for the case $(4,4)$:

\begin{theorem}\label{THEOREM MAIN (4,4)}
Suppose that $p \in P_{4,4}$ and $p$ is not sos. Then there exist three real quadrics $q_1,q_2,q_3 \in H_{4,2}$ intersecting in 8 (possibly complex) projective points $\gamma_1,\ldots,\gamma_8$ such that the values of $p$ on $\gamma_i$ certify that $p$ is not a sum of squares. More precisely, let $z_1,\ldots,z_8$ be affine
representatives of $\gamma_i$. Then there exists a real linear functional $\ell: H_{4,4} \rightarrow \mathbb{R}$ given by
$$\ell(f)=\sum \mu_if(z_i),$$
for some $\mu_i \in \mathbb{C}$ such that $\ell(r) \geq 0$ for all $r \in \Sigma_{4,4}$ and $\ell(p)<0$. Furthermore at most 2 of the points $\gamma_i$ are complex.
\end{theorem}

These theorems are proved at the end of Section \ref{NEW SECTION Proofs}. The cases $(3,6)$ and $(4,4)$ are quite similar, and we provide a unified presentation of the proofs. The main ingredient in the proofs is the Cayley-Bacharach theorem \cite{EGH}, which shows that the values of forms in $H_{3,3}$ (resp. $H_{4,2}$) on the points $z_i$ defined above are linearly related and this relation is unique. It was already observed by Hilbert in his original proof that the Cayley-Bacharach relations can be used to construct nonnegative polynomials that are not sums of squares. A modern exposition of Hilbert's construction along with generalizations is given by Reznick in \cite{Rez2}. We show that the Cayley-Bacharach relations are more than just a way of constructing examples and in the fact they are the fundamental reason that prevents sums of squares from filling out the entire psd cone. We note that for the cases where $\p=\sq$ the Cayley-Bacharach relations do not exist and it is possible to prove the equality of the psd and sos cones based on non-existence of the relations.

Complex zeroes of real forms come in conjugate pairs. In Section \ref{NEW SECTION Complex Points} we show how to exclude the cases of the intersection containing more than one conjugate pair of complex zeroes. We also show how to explicitly derive the inequalities $\ell$, given the Cayley-Bacharach relation. This is done for a fully real intersection case in Section \ref{SECTION FULLY REAL INTERSECTION} and in Section \ref{SECTION ONE COMPLEX PAIR} for the case of one conjugate pair of complex zeroes.

We also obtain the following interesting corollaries:

\begin{cor}\label{NEWCOR Pos 3,6}
Suppose that $p \in \Sigma_{3,6}$ lies on the boundary of the cone of sums of squares and $p$ is a strictly positive form. Then $p$ is a sum of $3$ squares and cannot be written as a sum of $2$ squares.
\end{cor}
And for the case $(4,4)$:
\begin{cor}\label{NEWCOR Pos 4,4}
Suppose that $p \in \Sigma_{4,4}$ lies on the boundary of the cone of sums of squares and $p$ is a strictly positive form. Then $p$ is a sum of $4$ squares and cannot be written as a sum of $3$ squares.
\end{cor}

The Corollaries \ref{NEWCOR Pos 3,6} and \ref{NEWCOR Pos 4,4} were used as a starting point to investigate the algebraic boundary of the cones $\Sigma_{3,6}$ and $\Sigma_{4,4}$ in \cite{Hilbert'sSOS}. Here we briefly note that sextics that are sums of three squares of cubics and quartics that are sums of four squares of quadratics form hypersurfaces in $H_{3,6}$ and $H_{4,4}$. One of the main results of \cite{Hilbert'sSOS} is establishing the degree of these hypersurfaces with a connection with K3 surfaces.

In Section \ref{SECTION GEOMETRY OF EVALUATIONS} we examine in detail the case of an arbitrary completely real transverse intersection of two cubics for the case $(3,6)$ and three quadratics for the case $(4,4)$. We provide a complete description of the differences between attainable values of psd forms and sos forms on the intersection points $z_i$. Let $E: \h \rightarrow \rmm$ be the evaluation map, sending $p \in \h$ to its values on $z_i$: $$E(p)=(p(z_1),\ldots,p(z_m)).$$ Here $m=9$ for the case $(3,6)$ and $m=8$ for the case $(4,4)$. Let $\rmm_{+}$ be the nonnegative orthant of $\rmm$, and let $\rmm_{++}$ denote the (open) strictly positive orthant. Let $P'$ and $Sq'$ be the images of $\p$ and $\sq$ under $E$. We show that in the cases $(3,6)$ and $(4,4)$ with $z_i$ coming from any completely real transverse intersection of two cubics or three quadratics the image $P'$ of $\p$ contains the positive orthant $\rmm_{++}$. In other words any combination of strictly values on the points $z_i$ is realizable by psd forms. However the Cayley-Bacharach relation forces restrictions on values of sos forms. We show the following (Theorem \ref{SOS=T_8}):

\begin{theorem}
We can choose affine representatives $z_1,\ldots,z_m$ for the projective points $\gamma_i$ so that the image $Sq'$ of the sos cone $\sq$ under $E$ is given by:
$$Sq'=\left\{(x_1,\ldots,x_m) \in \mathbb{R}^m_+ \hspace{5mm} \Big| \hspace{5mm} \sum_{i=1}^m\sqrt{x_i} \geq 2\sqrt{x_k} \hspace{3mm}\text{for all} \hspace{3mm} k \right\}.$$
\end{theorem}

If we intersect the images $P'$ and $Sq'$ with the hyperplane $L=\left\{x \in \rmm \hspace{3mm} | \hspace{3mm} \sum_{i=1}^m x_i=1 \right\}$, then $P' \cap L$ is essentially just a simplex since $\rmm_{++} \subset P'$, while $Sq'$ is a simplex with cut off corners.

The proofs for main theorems are obtained by analyzing the dual cone $\sq^*$. Let $K$ be a convex cone in a real vector space $V$. Its dual cone $K^*$ is defined as the set of all linear functionals in the dual space $V^*$ that are nonnegative on $K$:
$$K^*=\left\{\ell \in V^* \st \ell(x) \geq 0 \text{ for all } x \in K\right\}.$$

Let's consider the dual space $\h^*$ of linear functionals on $\h$. To every linear functional $\ell \in \h^*$ we can associate a quadratic form $Q_{\ell}$ defined on $\hd$ by setting
$$Q_{\ell}(f)=\ell(f^2) \text{  for all  } f \in \hd.$$

We classify the extreme rays of the dual cone $\sq^*$ which provides us with the description of all linear inequalities that define the sos cone. We prove the following theorems, which we think are interesting in themselves:

\begin{theorem}\label{NEWLABEL 3,6}
Suppose that $\ell$ spans an extreme ray of $\Sigma_{3,6}^*$. Then $\operatorname{rank} \ql=1$ or $\operatorname{rank} \ql=7$.
\end{theorem}

And for the case $(4,4)$:

\begin{theorem}\label{NEWLABEL 4,4}
Suppose that $\ell$ spans an extreme ray of $\Sigma_{4,4}^*$. Then $\operatorname{rank} \ql=1$ or $\operatorname{rank} \ql=6$.
\end{theorem}
We remark that in real analysis the functionals $\ell \in \h^*$ are represented by their values on the monomial basis and are called \textit{truncated moment sequences}. The matrix of the quadratic form $\ql$, when written with respect to the monomial basis of $\hd$ has several names: it is called \textit{the moment matrix} or \textit{Generalized Hankel matrix} in real analysis, and \textit{symmetric catalecticant matrix} in algebraic geometry. We prefer to keep a basis-free approach, but our results have interesting consequences when stated in terms of moment terminology.
\section{Dual Cones}

Let $\s$ be the vector space of real quadratic forms on $\hd$. We can view the dual space $\h^*$ as a subspace of $\s$ by identifying the linear functional $\ell \in \h^*$ with its quadratic form $Q_{\ell}$ defined by $\ql(f)=\ell(f^2)$. If we choose the basis of monomials for $\h$ then $\h^*$ is identified with the subspace of generalized Hankel matrices in $\s$ \cite{Rez3}. However, it is advantageous in our approach to not work with a fixed basis.

Let $\s^+$ be the cone of positive semidefinite forms in $\s$:
$$\s^+=\left\{ Q \in \s \hspace{1 mm} \mid \hspace{1 mm} Q(f) \geq 0  \hspace{3 mm} \text{for all} \hspace{3 mm} f \in \hd \right\}.$$

The following lemma is a well-known connection between  $\sq^*$ and $\s^+$, that allows sums of squares problems to be solved by semidefinite programming. Viewed with the monomial basis it says that $\sq^*$ is the intersection of $\s^+$ with the subspace of generalized Hankel matrices, thus $\sq^*$ is the Hankel spectrahedron.

\begin{lemma}\label{LEMMA dual of sos cone}
The cone $\sq^*$ is the section of the cone of psd matrices $\s^+$ with the subspace $\h^*$:
$$\sq^*=\s^+ \cap \h^*.$$
\end{lemma}
\begin{proof}
Suppose that $\ell \in \sq^*$. Then $\ell(f^2) \geq 0$ for all $f \in \hd$. By definition of $Q_{\ell}$ we see that it must be psd. Thus $\sq^* \subseteq \s^+ \cap \h^*$.

Now suppose that $Q_{\ell} \in \s^+ \cap \h^*$. Then it follows that $\ell(f^2) \geq 0$ for all $f \in \hd$ and thus $\ell \in \sq^*$. Thus $\s^+ \cap \h^* \subseteq \sq^*$ and the lemma follows.
\end{proof}

We now need a general lemma about extreme rays of sections of the cone of positive semidefinite forms. Let $S$ be the vector space of quadratic forms on a real vector space $V$. Let $S^+$ be the cone of psd forms in $S$. The following lemma is from \cite{RG} (Corollary 4), we provide a proof for completeness:
\begin{lemma}\label{LEMMA sections of psd matrix cone}
Let $L$ be a linear subspace of $S$ and let $K$ be the section of $S^+$ with $L$: $$K=S^+ \cap L.$$ Suppose that a quadratic form $Q$ spans an extreme ray of $K$. Then the kernel of $Q$ is maximal for all quadratic forms in $L$: if $P \in L$ and $\ker Q \subseteq \ker P$ then $P=\lambda Q$ for some $\lambda \in \mathbb{R}$.
\end{lemma}
\begin{proof}

Suppose not, so that there exists an extreme ray $Q$ of $K$ and a quadratic form $P \in L$ such that $\ker Q \subseteq \ker P$ and $P \neq \lambda Q$. Since $\ker Q \subseteq \ker P$ it follows that all eigenvectors of both $Q$ and $P$ corresponding to non-zero eigenvalues lie in the orthogonal complement $(\ker Q)^{\perp}$ of $\ker Q$. Furthermore, $Q$ is positive definite on $(\ker Q)^{\perp}$.

It follows that $Q$ and $P$ can be simultaneously diagonalized to matrices $Q'$ and $P'$ with the additional property that whenever the diagonal entry $Q'_{ii}$ is $0$ the corresponding entry $P'_{ii}$ is also $0$. Therefore, for sufficiently small $\epsilon \in \rr$ we have $Q+\epsilon P$ and $Q - \epsilon P$ are positive semidefinite and therefore $Q+\epsilon P,Q - \epsilon P \in K$. Thus $Q$ is not an extreme ray of $K$, which is a contradiction.


\end{proof}

Combining Lemma \ref{LEMMA dual of sos cone} and Lemma \ref{LEMMA sections of psd matrix cone} we obtain the following corollary, which will be a critical tool for describing the extreme rays of $\sq^*$:

\begin{cor}\label{COROLLARY no common zeroes}
Suppose that $Q$ spans an extreme ray of $\sq^*$. Then either $\operatorname{rank} Q =1$, or the forms in the kernel of $Q$ have no common projective zeroes, real or complex.
\end{cor}
\begin{proof}
Let $W \subset \hd$ be the kernel of $Q$ and suppose that the forms in $W$ have a common real zero $v \neq 0$. Let $\ell \in \h^*$ be the linear functional given by evaluation at $v$: $\ell(f)=f(v)$ for all $f \in \h$. Then $Q_{\ell}$ is a rank 1 positive semidefinite quadratic form and $\ker Q \subseteq \ker Q_{\ell}$. By Lemma \ref{LEMMA sections of psd matrix cone} it follows that $Q=\lambda Q_{\ell}$ and thus $Q$ has rank 1.

Now suppose that the forms in $W$ have a common nonreal zero $z \neq 0$.  Let $\ell \in \h^*$ be the linear functional given by taking the real part of the value at $z$: $\ell(f)=\operatorname{Re} f(z)$ for all $f \in \h$. It is easy to check that the kernel of $Q_{\ell}$ includes all forms that vanish at $z$ and therefore $W \subseteq \ker Q_{\ell}$. Therefore by applying Lemma \ref{LEMMA sections of psd matrix cone} we again see that $Q=\lambda Q_{\ell}$. However, we claim that $Q_{\ell}$ is not a psd form.

The quadratic form $Q_{\ell}$ is given by $Q_{\ell}(f)=\operatorname{Re} f^2(z)$ for $f \in \hd$. However, there exists $f \in \hd$ such that $f(z)$ is purely imaginary and therefore $Q_{\ell}(f)<0$. The Corollary now follows.

\end{proof}

We note that if we can find a nonzero psd quadratic form $Q_{\ell}$ such that the forms in its kernel $\wl$ have no common real zeroes then $\ell$ will indeed provide a linear inequality that holds for all sos forms but fails for some psd forms. Since $\ql$ is psd, we know that $\ell \in \sq^*$ and we need to construct a nonnegative $f \in \h$ such that $\ell(f)<0.$  Since forms in $\wl$ have no common real zeroes we can find $f_i \in \wl$ such that $q=\sum_i f_i^2$ is strictly positive. We have $\ql(f_i)=\ell(f_i^2)=0$ for all $i$. Therefore $\ell(q)=0$ and $q$ is strictly positive on the unit sphere. For sufficiently small $\epsilon >0$ we know that $f=q-\epsilon(x_1^2+\ldots+x_n^2)^d$ is nonnegative. On the other hand we have $\ell(f)=-\epsilon \ell((x_1^2+\ldots+x_n^2)^d)<0$.

We will also need the following classification of all rank 1 forms in $\h^*$. For $v \in \rn$ let $\ell_v$ be the linear functional in $\h^*$ given by evaluation at $v$: $$\ell_v(f)=f(v) \text{  for  } f \in \h,$$ and let $Q_v$ be the quadratic form associated to $\ell_v$: $Q_v(f)=f^2(v).$ In this case we say that $Q_v$ (or $\ell_v$) corresponds to point evaluation. Note that the inequalities $\ell_v \geq 0$ are the defining inequalities of $\p^*$. The following lemma shows that all rank 1 forms in $\h^*$ correspond to point evaluations. Since we are interested in the inequalities that are valid on $\sq$ but not valid on $\p$ it allows us to disregard rank 1 forms $Q \in \h^*$ .

\begin{lemma}\label{LEMMA rank 1 quad forms}
Suppose that $Q$ is a rank 1 form in $\h^*$. Then $Q=\lambda Q_{v}$ for some $v \in \rn$ and $\lambda \in \mathbb{R}$.
\end{lemma}
\begin{proof}
Let $Q$ be a rank 1 form in $\h^*$. Then $Q(f)=\lambda s^2(f)$ for some linear functional $s \in \hd^*$. Therefore it suffices to show that if $Q=s^2(f)$ for some $s \in \hd^*$ then $Q=Q_v$ for some $v \in \rn$.

Since $Q \in \hd^*$ we know that $Q$ is defined by $Q(f)=\ell(f^2)$ for a linear functional $\ell \in \h^*$ and therefore $\ell(f^2)=s^2(f)$ for all $f \in \hd$. We have $Q(f+g)=\ell((f+g)^2)=\ell(f^2)+2\ell(fg)+\ell(g^2)=(s(f)+s(g))^2=s^2(f)+2s(f)s(g)+s^2(g)$ and it follows that $\ell(fg)=s(f)s(g)$ for all $f,g \in \hd.$

Let $x^{\alpha}$ denote the monomial $x_1^{\alpha_1}\cdots x_n^{\alpha_n}$. If we take monomials $x^{\alpha}$, $x^{\beta}$, $x^{\gamma}$, $x^{\delta}$ in $\hd$ such that $x^{\alpha}x^{\beta}=x^{\gamma}x^{\delta}$ then we must have $s(x^{\alpha})s(x^\beta)=s(x^\gamma)s(x^\delta).$

Suppose that $s(x_i^d)=0$ for all $i$. Then we see that $s(x_i^{d-1}x_j)^2=s(x_i^d)s(x_{i}^{d-2}x_j^2)=0$ and continuing in similar fashion we have $s(x^{\alpha})=0$ for all monomials. Then $\ell$ is the zero functional and $Q$ does not have rank one. Contradiction.

We may assume without loss of generality that $s(x_1^d) \neq 0$. Since we are interested in $\ell(f^2)=s^2(f)$ we can work with $-s$, if necessary, and thus we may assume that $s(x_1^d)>0$. Let $s_i=s(x_1^{d-1}x_i)$ for $1 \leq i \leq n$. We will express $s(x^{\alpha})$ in terms of $s_i$ for all $x^\alpha \in \hd$. Since $(x_1^d)(x_1^{d-2}x_ix_j)=(x_1^{d-1}x_i)(x_1^{d-1}x_j)$ we have $s(x_1^{d-2}x_ix_j)=s_is_j/s_1$. Continuing in this fashion we find that $$s(x_1^{\alpha_1}\cdots x_{n}^{\alpha_n})=\frac{s_2^{\alpha_2}\cdots s_n^{\alpha_n}}{s_1^{d-1-\alpha_1}}.$$

Now let $v \in \rn$ be the following vector $$v=(s_1^{1/d},s_1^{-(d-1)/{d}}s_2, \ldots,s_1^{-(d-1)/d}s_n).$$
Let $s_v$ be the linear operator on $\hd$ defined by evaluating a form at $v$: $s_v(f)=f(v)$. Then we have $s_v(x_1^{d-1}x_i)=s_i$ and $$s_{v}(x_1^{\alpha_1}\cdots x_n^{\alpha_n})=s_2^{\alpha_2}\cdots s_n^{\alpha_n}s_1^{\alpha_1/d-(d-1)(d-\alpha_1)/d}=\frac{s_2^{\alpha_2}\cdots s_n^{\alpha_n}}{s_1^{d-1-\alpha_1}}.$$

Since $s$ agrees with $s_v$ on monomials it follows that $s=s_v$ and thus $\ell(f^2)=s^2(f)=f(v)^2=f^2(v)$. Therefore $\ell$ indeed corresponds to points evaluation and we are done.
\end{proof}

\subsection{Kernels of Extreme Rays}Let $Q_{\ell}$ span an extreme ray of $\sq^*$ that does not correspond to point evaluation. Let $W_{\ell}$ be the kernel of $Q_{\ell}$ and let $J(\ell)$ be the ideal generated by $\wl$. By Corollary \ref{COROLLARY no common zeroes} and Lemma \ref{LEMMA rank 1 quad forms} we know that the forms in $\wl$ have no common projective zeroes real or complex, i.e. $\vc(\wl)=\emptyset$. We now investigate the kernel $\wl$ further.


Forms $p_1,\ldots,p_n \in \hd$ are said to form a \textit{sequence of parameters} if they have no common projective complex zeroes: $$\vc(p_1,\ldots,p_n)=\emptyset.$$
It follows that we can find a sequence of parameters $p_1, \ldots,p_n \in \wl$. Let $I$ be the ideal generated by $p_1,\ldots,p_n$. We will need the following theorem (special case of \cite[Theorem CB8]{EGH}):

\begin{theorem}\label{THEOREM Gorenstein}
Suppose that $p_1,\ldots,p_n \in \hd$ are a sequence of parameters and let $I$ be the ideal generated by $p_1,\ldots,p_n$ in $\mathbb{C} [x_1,\dots,x_n]$. Then $I$ is a Gorenstein ideal with socle of degree $n(d-1)$.
\end{theorem}
We also prove a simple but very useful characterization of kernels of forms $\ql \in \h^*$:
\begin{lemma}\label{LEMMA kernel structure}
Let $Q_{\ell}$ be a quadratic form in $\h^*$. Then $p \in \wl$ if and only if $\ell(pq)=0$ for all $q \in \hd$.
\end{lemma}
\begin{proof}
In order to investigate $\wl$ need to define the associated bilinear form $B_{\ell}$:
$$
B_{\ell}(p,q)=\frac{\ql(p+q)-\ql(p)-\ql(q)}{2} \text{   for   } p,q \in \hd.
$$
By definition of $\ql$ we have $\ql(p)=\ell(p^2).$ Therefore it follows that
\begin{equation}\label{EQUATION kernel of B_l}B_{\ell}(p,q)=\ell(pq).\end{equation}
A form $p \in \hd$ is in the kernel of $\ql$ if and only if $\bl(p,q)=0$ for all $q \in \hd$. Using \eqref{EQUATION kernel of B_l} the lemma follows.
\end{proof}


We are now in position to prove Theorems \ref{NEWLABEL 3,6} and \ref{NEWLABEL 4,4}, which we restate in a unified way:
\begin{theorem}\label{COROLLARY Dimension Kernel}
Suppose that $\ell$ is an extreme ray of $\sq^*$ in the cases $(3,6)$ and $(4,4)$ and $\ell$ does not correspond to point evaluation. Then rank of $\ql$ is equal to $\dim \hd - n$.
\end{theorem}
\begin{proof}

Let $p_1, \ldots,p_n$ be a sequence of parameters in $\wl$ and let $I$ be the ideal generated by $p_1,\ldots,p_n$. We claim that $\wl=I_d$, or in other words, linear combinations of $p_1,\ldots,p_n$ generate $\wl$. We note that this claim implies the desired Corollary, since it shows that the kernel of $\ql$ has dimension exactly $n$.

By Theorem \ref{THEOREM Gorenstein} we know that the socle of $I$ has degree $n(d-1)=2d$ in the cases $(3,6)$ and $(4,4)$. Suppose that $\wl$ is strictly larger than $I_d$. The ideal $I$ is Gorenstein with socle of degree $2d$, and hence $J(\ell)_{2d}$ is strictly larger than $I_{2d}$, which means that $J(\ell)_{2d}=\h$.

It follows from Lemma \ref{LEMMA kernel structure} that $$\ell(f)=0 \hspace{5mm} \text{for all} \hspace{5mm} f \in J(\ell)_{2d}.$$ Therefore $\ell$ is the zero linear functional, which is a contradiction.
\end{proof}

Given an extreme ray $\ql$ of $\sq^*$ that does not correspond to point evaluation we can pass to its kernel $\wl$ and in the cases $(3,6)$ and $(4,4)$ the kernel $\wl$ has dimension exactly $n$ and further $\vc(\wl)=\emptyset$. It follows from Theorem \ref{THEOREM Gorenstein} that an $n$-dimensional subspace $W$ with $\vc(W)=\emptyset$ uniquely determines (up to a constant multiple) the linear functional $\ell$ such that the kernel of $\ql$ is $W$. The linear functional $\ell$ is the unique linear functional vanishing on the degree $2d$ part $\langle W \rangle_{2d}$ of the ideal generated by $W$. This correspondence is a special case of the \textit{global residue map} \cite[\S 1.6]{CD}.

Therefore instead of directly studying the extreme rays $\ell$ of $\sq^*$ we can look instead for $n$-dimensional subspaces $W$ of $\hd$, with $\mathcal{V}_{\co}(W)=\emptyset$, whose corresponding linear functionals are extreme rays of $\sq^*$. The linear functionals $\ell \in \sq^*$ have the defining property of being nonnegative on squares. In order to see when a subspace $W$ of $\hd$ gives rise to an extreme ray of $\sq^*$ we need to get a handle on the linear functional $\ell \in \h^*$ that $W$ defines. We do this by passing to point evaluations. We need the following general Lemma, which allows us to extract a transverse zero-dimensional intersection from forms in $W$.

\begin{lemma}\label{NEWLEMMA Bertini}
Suppose that $p_1,\ldots,p_n \in \hd$ are a sequence of parameters. Then there exist $f_1, \ldots,f_{n-1}$ in the real linear span of $p_i$ such that the forms $f_1, \ldots, f_{n-1}$ intersect transversely in $d^{n-1}$ (possibly complex) points.
\end{lemma}

\begin{proof}
Let $W$ be the linear span of $p_1, \ldots, p_n$ with complex coefficients. We begin by showing that there exist linear combinations $f_1,\ldots, f_{n-1} \in W$ such that $f_1, \ldots f_{n-1}$ intersect transversely in $\mathbb{CP}^{n-1}$.

By Bertini's theorem a general form in $W$ is smooth. Let $f_1$ be such a form. Let $V_1$ be the smooth variety defined by $f_1$ and let $W_1$ be a subspace of $W$ complementary to $f_1$. Then $W_1$ defines a linear system of divisors on $V_1$ and by Bertini's Theorem the intersection of $V_1$ with a general element of $W_1$ is a smooth variety of dimension $n-2$. Let $f_2$ be such an element of $W_1$. Now we can let $V_2$ be the smooth variety defined by $f_1$ and $f_2$, let $W_2$ be the complementary subspace to $f_1$ and $f_2$ and repeatedly apply Bertini's Theorem until we get a $0$-dimensional smooth intersection. Hence the forms $f_1,\dots,f_{n-1}$ we constructed intersect transversely.

Now we argue that there exist \textit{real} linear combinations $f_1,\ldots, f_{n-1}$ which intersect transversely. Suppose not and let $f_i=\sum_{j=1}^{n}\alpha_{ij}p_j$. Then for all $\alpha_{ij} \in \mathbb{R}$ the forms $f_i$ do not intersect transversely. This is an algebraic conditions on the coefficients $\alpha_{ij}$, given by vanishing of some polynomials in the variables $\alpha_{ij}$. However, if a polynomial vanishes on all real points then it must be identically zero. Therefore, no complex linear combinations of $p_i$ intersect transversely, which is a contradiction.
\end{proof}

Now suppose that we have an $n$-dimensional subspace $W$ of $\hd$ with $\vc(W)=\emptyset$ and we locate $f_1,\dots,f_{n-1} \in W$ that intersect transversely. Let $s=d^{n-1}$ and let $\Gamma$ be the complex projective variety defined by $f_1,\dots,f_{n-1}$:$$\Gamma=\mathcal{V}_{\co}(f_1,\ldots,f_{n-1})=\{\gamma_1, \dots , \gamma_s\}\subset \co\mathbb{P}^{n-1}$$. Let $S=\{z_1,\ldots,z_s\}$ be a set of affine representatives for projective points $\gamma_i$. The functional $\ell$ determined by $W$ is the unique linear functional vanishing on $\langle W \rangle_{2d}$. In particular $\ell$ vanishes on $\langle f_1,\dots,f_{n-1} \rangle_{2d}$. Since the forms $f_1,\dots,f_{n-1}$ intersect transversely, the ideal generated by $f_i$ is radical \cite{Harris Book}. It follows therefore that $\ell$ can be expressed as a linear combination of evaluations at points $v_i$:
$$\ell=\sum_{i=1}^{s} \mu_i\ell_{v_i};\qquad \ell(p)=\sum_{i=1}^{s} \mu_ip(v_i), \quad p \in \h.$$

The coefficients $\mu_i$ are determined uniquely from any form $f_n$ so that $f_1, \dots, f_n$ form a basis of $W$. In order to see how this occurs we need to introduce Cayley-Bacharach relations.
\subsection{Cayley-Bacharach Relations}

We now recall the Cayley-Bacharach theorem as applicable to ternary cubics and quaternary quadrics \cite[Theorem CB6]{EGH}:
\begin{lemma} \label{NEWLEMMA CB}
For the cases $(3,6)$ and $(4,4)$ let $f_1,\dots,f_{n-1} \in \hd$ be forms intersecting transversely in $s=d^{n-1}$ complex projective points $\gamma_1, \ldots, \gamma_s$. Let $z_1, \ldots, z_s$ be affine representatives of the projective points $\gamma_i$. Then there is a unique linear relation on the values of any form in $\hd$ on $z_i$:
$$u_1p(z_1)+\ldots+u_sp(z_s)=0 \quad \text{for all}\quad p \in \hd,$$
with nonzero $u_i \in \mathbb{C}.$
\end{lemma}

As we will see later evaluation on transverse intersections will be enough to distinguish nonnegative forms from sums of squares. Before we proceed with that we would like to show explicitly the geometry of values on transverse intersections, when all of the intersections points are real.
\section{Cones of Point Evaluations}\label{SECTION GEOMETRY OF EVALUATIONS}
Since the geometry of the cases $(3,6)$ and $(4,4)$ is very similar we will give a unified presentation.
For these cases let $f_1,\dots, f_{n-1}$ be forms in $\hd$ intersecting transversely in $s=d^{n-1}$ real projective points $\gamma_1,\dots,\gamma_s$. Let $v_1, \ldots, v_s \in \mathbb{R}^n$ be arbitrary nonzero affine representatives for $\gamma_1,\ldots,\gamma_s$ with $v_i$ corresponding to $\gamma_i$. Then by Lemma \ref{NEWLEMMA CB} there exists a unique linear relation $$u_1p(z_1)+\ldots+u_sp(z_s)=0 \quad \text{for all}\quad p \in \hd$$ and since we have all points $v_i \in \mathbb{R}^n$ coming from intersection of real forms, all the coefficients $u_i$ must be real.


We first look at general real zero-dimensional intersections. Suppose that $\Gamma =\{\gamma_1,\ldots,\gamma_m\}$ are real projective points that can be given as the complete set of common real zeroes of some forms $f_1,\ldots,f_k \in \hd$: $$\Gamma=\mathcal{V}_{\rr}(f_1,\ldots,f_k).$$

Let $S=\{v_1,\ldots,v_m\} \subset \rn$ be a set of affine representatives for $\gamma_i$. Let $E$ be the evaluation map that sends $p \in \h$  to its values on the points $v_i$:
$$E:\h \longrightarrow \rmm, \hspace{8mm} E(p)=(p(v_1),\ldots,p(v_m)).$$

Note that $E$ is defined on forms of degree $2d$. Let $P'$ and $Sq'$ be the images of $\p$ and $\sq$ under $E$ respectively and let $H'$ be the image of $\h$. We observe that $H'$ does not have to be all of $\rr^m$, since the values of forms in $\h$ on points $v_i$ may be linearly dependent. Since we are evaluating nonnegative forms it follows that $P'$ lies inside the intersection of $H'$ and $\rr^m_+$: $$P' \subseteq H' \cap \rr^m_+.$$
The following theorem shows that this inclusion is almost an equality.
\begin{theorem}\label{THM Values of Nonegative Forms}
Let $\rr^m_{++}$ be the positive orthant of $\rr^m$. The intersection of $H'$ with the positive orthant is contained in $P'$: $$H' \cap \rr^m_{++} \subset P'.$$
\end{theorem}

\begin{proof}
Let $s=(s_1,\ldots, s_m)$ be a point in the intersection of $H'$ and $\rr^m_{++}$. Since $s \in H$ there exists a form $p \in \h$ such that $p(v_i)=s_i$. Let $g=f_1^2+\ldots+f_k^2$. We claim that for large enough $\lambda \in \rr$ the form $\bar{p}=p+\lambda g$ will be nonnegative, and since each $f_i$ is zero on $S$ we will also have $E\left(\bar{p}\right)=s$.

By homogeneity of $\bar{p}$ it suffices to show that it is nonnegative on the unit sphere $\sph$. Furthermore, we may assume that the evaluation points $v_i$ lie on the unit sphere. Since we are dealing with forms, evaluation on the points outside of the unit sphere amounts to rescaling of the values on $\sph$.

Let $B_{\epsilon}(S)$ be the open epsilon neighborhood of $S$ in the unit sphere $\sph$. Since $p(v_i)>0$ for all $i$, it follows that for sufficiently small $\epsilon$ the form $p$ is strictly positive on $B_{\epsilon}(S)$: $$p(x) >0 \quad \text{for all} \quad x \in B_{\epsilon}(S).$$
\noindent The complement of $B_{\epsilon}(S)$ in $\sph$ is compact, and therefore we can let $m_1$ be the minimum of $g$ and $m_2$ be the minimum of $p$ on $\sph \setminus B_{\epsilon}(S)$. If $m_2\geq 0$ then $p$ itself was nonnegative and we are done. Therefore, we may assume $m_2 <0$. We also note that since $g$ vanishes on $S$ only, it follows that $m_1$ is strictly positive.

Now let $\lambda \geq -\frac{m_2}{m_1}$. The form $\bar{p}=p+\lambda g$ is positive on $B_{\epsilon}(S)$. By construction of $B_{\epsilon}(S)$ we also see that the minimum of $\bar{p}$ on the complement of $B_{\epsilon}(S)$ is at least $0$. Therefore $\bar{p}$ is nonnegative on the unit sphere $\sph$ and we are done.
\end{proof}



We obtain the following corollary for the cases $(3,6)$ and $(4,4)$:
\begin{cor}
Suppose that $\Gamma$ comes from transverse intersection of two ternary cubics (resp. 3 quaternary quadrics) then $\mathbb{R}^9_{++} \subset P'$ (resp. $\mathbb{R}^8_{++} \subset P'$).
\end{cor}
\begin{proof}
We need to show that in our two cases the cone $P'$ is full dimensional. This happens if and only if the values on the points $v_i$ are linearly independent for forms in $H_{3,6}$ (resp. $H_{4,4}$). This is an easy special case of Cayley-Bacharach Theorem \cite[Theorem CB6]{EGH}.
\end{proof}

We now show how the presence of a Cayley-Bacharach relation impacts the values attainable by sums of squares. Suppose now that the points $v_1,\ldots,v_m \in \rn$  are such that there exists a unique Cayley-Bacharach relation satisfied by all forms $p \in \hd$: $u_1p(v_1)+\ldots+u_mp(v_m)=0$ with nonzero coefficients $u_i \in \mathbb{R}$.

\indent We first describe $Sq'$ if the affine representatives are chosen so that the coefficients in the Cayley-Bacharach relation have absolute value 1. Let $w_i=|u_i|^{1/d}v_i$. Then $p(w_i)=|u_i|p(v_i)$ for all $f \in \hd$. Thus we see that the values of forms in $\hd$ on $w_i$ satisfy a unique relation $a_1f(w_1)+\ldots+a_mf(w_m)=0$ with $a_i=\pm 1$. Now redefine the map $E$ using this particular set of affine representatives $w_i$ and let $Sq'$ be the image of $\sq$ under $E$.

Let $T_m$ be the subset of the nonnegative orthant $\mathbb{R}^m_+$ defined by the following $m$ inequalities: $$T_m=\left\{(x_1,\ldots,x_m) \in \mathbb{R}^m_+ \hspace{5mm} \Big| \hspace{5mm} \sum_{i=1}^m\sqrt{x_i} \geq 2\sqrt{x_k} \hspace{3mm}\text{for all} \hspace{3mm} k \right\}.$$
\begin{lemma}\label{TTn}
The set $T_m$ is a closed convex cone. Moreover, $T_m$ is the convex hull of the points $x=(x_1,\ldots,x_m) \in \rmm_+$ where $\sum_{i=1}^m\sqrt{x_i} = 2\sqrt{x_k}$ for some $k$.
\end{lemma}
\begin{proof}
$T_m$ is defined as a subset of $\rmm$ by the following $2m$ inequalities: $x_k \geq 0$  and  $\sqrt{x_1}+\ldots+\sqrt{x_m} \geq 2\sqrt{x_k}$ for all $k$. Therefore it is clear that $T_m$ is a closed set.

For $x=(x_1,\ldots,x_m) \in \rmm_+$ let $\norm{x}_{1/2}$ denote the $L^{1/2}$ norm of $x$:
$$\norm{x}_{1/2}=(\sqrt{x_1}+\ldots+\sqrt{x_m})^2.$$
We can restate inequalities of $T_m$ as $x_k \geq 0$ and $\norm{x}_{1/2} \geq 4x_k$ for all $k$. Now suppose that $x,y \in T_m$ and let $z=\lambda x+ (1- \lambda) y$ for some $0 \leq \lambda \leq 1$. It is clear that $z_k \geq 0$ for all $k$. It is known by the Minkowski inequality (\cite{HLP} p.30)  that $L^{1/2}$ norm is a concave function: $\norm{\lambda x+ (1- \lambda) y}_{1/2} \geq \lambda \norm{x}_{1/2}+(1-\lambda)\norm{y}_{1/2}$. Therefore $$\norm{z}_{1/2} \geq \lambda \norm{x}_{1/2}+(1-\lambda)\norm{y}_{1/2} \geq 4\lambda x_k +4(1-\lambda)y_k=4z_k.$$
Therefore $T_m$ is a convex cone.

To show that $T_m$ is the convex hull of the points where $\norm{x}_{1/2}=4x_k$ for some $k$ we proceed by induction. The base case $m=2$ is simple since $T_2$ is just a ray spanned by the point $(1,1)$. For the induction step we observe that any convex set is the convex hull of its boundary. For any point in the boundary of $T_m$ one of the defining $2m$ inequalities must be sharp. If a point $x$ is in the boundary of $T_m$ and $x_i \neq 0$ for all $i$ then the inequalities $x_i \geq 0$ are not sharp at $x$ and therefore the inequality  $\norm{x}_{1/2} \geq 4x_k$ must be sharp for some $k$ and we are done.

If $x_i=0$ for some $i$ then the point $x$ lies in the set $T_{m-1}$ in the subspace spanned by the $m-1$ standard basis vectors excluding $e_i$ and we are done by induction.
\end{proof}

\begin{theorem}\label{SOS=T_8}
With the choices of affine representatives $w_i$, so that the coefficients in the unique Cayley-Bacharach relation are of absolute value 1, we have $Sq'=T_m$.
\end{theorem}
\begin{proof}
By slight abuse of notation we will also use $E$ as the evaluation map at $w_i$ for forms in $\hd$. Let $a=(a_1,\ldots,a_m)$ be the vector of coefficients in the Cayley-Bacharach relation $a_1p(w_1)+\ldots+a_mp(w_m)=0$ with $a_i=\pm 1$ and $f \in H_{n,d}$.

By uniqueness of the Cayley-Bacharach relation by know that $L=E(H_{n,d})$ is the hyperplane in $\mathbb{R}^m$ perpendicular to $a$. To show that $Sq' \subseteq T_m$ it suffices to show that $E(q^2) \in T_m$ for any $q \in H_{n,d}$. Let $s=E(q)$ and $t=E(q^2)$. We know that $E(q^2)=(t_1,\ldots,t_m)=(s_1^2,\ldots,s_m^2)$.
By the Cayley-Bacharach relation we have $a_1s_1+\ldots+a_ms_m=0$ with $a_i=\pm 1$. Without loss of generality, we may assume that $s_1$ has the maximal absolute value among $s_i$. Multiplying the Cayley-Bacharach relation by $-1$, if necessary, we can make $a_1=-1$. Then we have $s_1=a_2s_2+\ldots+a_ms_m$. We can now write $\sqrt{t_1}=\pm \sqrt{t_2} \pm \sqrt{t_3}\pm \ldots \pm \sqrt{t_m}$ with the exact signs depending on $a_i$ and signs of $s_i$. Therefore we see that $2\sqrt{t_1} \leq \sqrt{t_1}+\ldots+\sqrt{t_m}.$ Since $s_1$ had the largest absolute value among $s_i$ it follows that $Sq' \subseteq T_{m}$.

To show the reverse inclusion $T_m \subseteq Sq'$ we use Lemma \ref{TTn}. It suffices to show that all points in $x \in T_m$ with $2\sqrt{x_k}=\sqrt{x_1}+\ldots+\sqrt{x_m}$ for some $k$, are also in $Sq'$. Without loss of generality may assume that $k=1$ and we have $\sqrt{x_1}=\sqrt{x_2}+\ldots+\sqrt{x_m}.$ Let $y=(y_1,\ldots,y_m)$ with $y_1=-\sqrt{x_1}/a_1$ and $y_i=\sqrt{x_i}/a_i$ for $2 \leq i \leq m$. It follows that $a_1y_1+\ldots+a_my_m=0$. Therefore $y \in E(H_{n,d})$ and $y=E(q)$ for some quadratic form $q$. Then $E(q^2)=x$ and we are done.
\end{proof}

Note that this already proves Hilbert's Theorem that there exist nonnegative polynomials that are not sums of squares for the cases $(3,6)$ and $(4,4)$. In fact Hilbert's proof, by different methods, established that the standard basis vectors are not in $Sq'$, while we provide a complete description of $Sq'$.

We now describe what happens if we do not rescale the affine representatives and the coefficients in the Cayley-Bacharach relation are arbitrary real numbers. Suppose that the unique Cayley-Bacharach relation satisfied by all forms $f \in \hd$ is given by: $u_1f(v_1)+\ldots+u_mf(v_m)=0$ with nonzero coefficients $u_i$. Let $E$ be the evaluation map at the points $v_i$ and let $Sq'$ be the image of $\sq$ under $E$.

\begin{cor}
The cone $Sq'$ is the subset of $\mathbb{R}^m_+$ satisfying the following $m$ inequalities:
$$|u_1|\sqrt{x_1}+\ldots+|u_m|\sqrt{x_m} \geq 2|u_k|\sqrt{x_k},$$
for all $1 \leq k \leq m$.
\end{cor}
\begin{proof}
Let $a \in \mathbb{R}^m$ be a vector with $a_i=u_i/|u_i|$ and let $D$ be the diagonal $m \times m$  matrix with $D_{ii}=|u_i|$. Let $L_a$ be the hyperplane of vectors in $\mathbb{R}^m$ perpendicular to $a$ and $L_u$ be the hyperplane of vectors perpendicular to $u$. The linear transformation $\bar{D}$ sending $x \in \mathbb{R}^m$ to $Dx$ sends $L_u$ to $L_a$. \\
\indent Since the Cayley-Bacharach relation is unique it follows that $Sq'$ is the convex hull of the points $(y_1^2,\ldots,y_m^2)$ with $y=(y_1,\ldots,y_m) \in L_u$. We have shown in Theorem \ref{SOS=T_8} that the convex hull of squares from $L_a$ is $T_m$. Since $\bar{D}$ sends $L_u$ to $L_a$ it follows that $\bar{D}^2$ sends $Sq'$ to $T_m$.\\
\indent By Lemma \ref{TTn} we know that $T_m$ is the set of $x \in \mathbb{R}^m_+$ satisfying inequalities $\sqrt{x_1}+\ldots+\sqrt{x_m} \geq 2\sqrt{x_k}$
for all $1 \leq k \leq m$. Now suppose that $x=D^2y$ with $x \in T_m$ and $y \in Sq'$. Then $x_i=|u_i|^2y_i$ and it follows that $y$ satisfies $|u_1|\sqrt{x_1}+\ldots+|u_m|\sqrt{x_m} \geq 2|u_k|\sqrt{x_k}$ for all $1 \leq k \leq m$. Since $D^2$ is an invertible linear transformation ($u_i \neq 0$) it follows that all $y$ satisfying these inequalities are in $Sq'$.
\end{proof}

\section{Structure of Extreme Rays of $\Sigma_{3,6}^*$ and $\Sigma_{4,4}^*$}
We now return to the study of extreme rays of $\sq^*$ for the cases $(3,6)$ and $(4,4)$. Let $W$ be an $n$-dimensional linear subspace of $\hd$ with $\vc(W)=\emptyset$. Let $f_1,\dots,f_{n-1} \in W$ be forms intersecting transversely in $s=d^{n-1}$ points $\gamma_1,\dots, \gamma_s$. Let $z_1,\ldots,z_s$ be affine representatives for points $\gamma_i$. By Lemma \ref{NEWLEMMA CB} there is a unique linear relation for values of forms in $\hd$ on the points $z_i$:
$u_1f(z_1)+\dots+u_sf(z_s)=0,$ for all $f \in \hd.$ The unique (up to a constant multiple) linear functional $\ell$ vanishing on $\langle W \rangle_{2d}$ can be written as a linear combination of point evaluations on the points $z_i$:
$$\ell=\sum_{i=1}^s \mu_i \ell_{z_i}, \quad \text{for some} \quad \mu_i\in \mathbb{C} .$$

Let $f_n$ be a form in $W$ such that $f_1,\dots,f_n$ form a basis of $W$. Note that the values $f_n(z_i)$ are the same regardless of which $f_n \in W$ we choose. We now explain how to determine the coefficients $\mu_i$ from the knowledge of the Cayley-Bacharach coefficients $u_i$ and the values $f_n(z_i)$.

\begin{lemma}\label{NEWLEMMA Relation Structure}
$$\mu_i=\frac{u_i}{f_n(z_i)}; \qquad \ell=\sum_{i=1}^s \frac{u_i}{f_n(z_i)}\ell_{z_i}\quad i=1 \dots s.$$
\end{lemma}

\begin{proof}
Let $\ell$ be defined as above. We need to show that $\ell$ vanishes on all forms in $\langle W \rangle_{2d}$. We observe that for any form $q \in \hd$ we have $$\ell(f_1q)=\dots=\ell(f_{n-1}q)=0$$ since $\ell$ is defined by values at common zeroes of $f_1, \dots,f_{n-1}$. Also $$\ell(f_nq)=\sum_{i=1}^s \frac{u_i}{f_n(z_i)}f_n(z_i)q(z_i)=\sum_{i=1}^s u_i q(z_i)=0,$$ by the Cayley-Bacharach relation.

\end{proof}

Since the forms $f_1,\dots f_{n-1}$ are real, the set $\Gamma=\{\gamma_1,\ldots,\gamma_s\}$ is invariant under conjugation. Hence we can choose affine representatives $z_i$ so that the set $S=\{z_1,\ldots,z_s\}$ is invariant under conjugation. By uniqueness of the Cayley-Bacharach relation it follows that if $z_i=\bar{z}_j$ then $u_i=\bar{u}_j$. We now show that if the functional $\ell$ is nonnegative on squares, then we can restrict the number of complex points $z_i$, forcing most of the intersection points to be real.

\subsection{Number of Complex Points}\label{NEW SECTION Complex Points}

Suppose that $S$ is a finite set of points in $\co^n$ that is invariant under conjugation: $\bar{S}=S$. Let $S$ be given by $S=\{r_1,\ldots,r_k, z_1,\ldots,z_m, \bar{z}_1,\ldots,\bar{z}_m\}$ with $r_i \in \rr^n$ and $z_i,\bar{z}_i \in \co^n$, $z_i\neq \bar{z}_i$. Let $\ell: \h \rightarrow \rr$ be a linear functional given as a combination of evaluations on $S$:
$$\ell(p)=\sum_{i=1}^k \lambda_ip(r_i)+\sum_{i=1}^m \left(\mu_ip(z_i)+\bar{\mu}_ip(\bar{z}_i)\right),\quad p \in \h$$

\noindent with $\lambda_i \in \rr,$ $\mu_i \in \co$ and $\alpha_i,\mu_i \neq 0$.

Let $E_{\rr}: \hd \rightarrow \rr^{k+2m}$ be the real evaluation projection of forms in $\hd$ given by
$$E_{\rr}(p)=\left(p(r_1),\ldots,p(r_k),\operatorname{Re} p(z_1),\operatorname{Im} p(z_1),\ldots,\operatorname{Re} p(z_m),\operatorname{Im} p(z_m)\right), \quad p \in \hd.$$

\noindent Let $c$ be the dimension of the image of $E_{\rr}$: $$c=\dim E_{\rr}(\hd).$$

\begin{lemma}Suppose that the quadratic form $\ql$ is positive semidefinite. Then $c \leq k+m$.
\end{lemma}
\begin{proof}
The quadratic form $\ql: \hd \rightarrow \rr$ is defined by
$$\ql(q)=\sum_{i=1}^k \lambda_iq^2(r_i)+\sum_{i=1}^m \left(\mu_iq^2(z_i)+\bar{\mu}_iq^2(\bar{z}_i)\right), \quad q \in \hd.$$

\noindent Let $\bar{Q}_{\ell}$ be the quadratic form on $\co^{k+2m}$ given by: $$\sum_{i=1}^k \lambda_ix_i^2+\sum_{i=1}^m \mu_i\left(x_{2i-1}+\sqrt{-1}x_{2i} \right)^2+\bar{\mu}_i\left(x_{2i-1}-\sqrt{-1}x_{2i}\right)^2.$$

\noindent By its definition, the form $\ql$ is a composition on $E_{\rr}$ and $\bar{Q}_{\ell}$:

$$\ql=\bar{Q}_{\ell}\circ E_{\rr}.$$

Each of the $2$ variable blocks $\mu_i\left(x_{2i-1}+\sqrt{-1}x_{2i} \right)^2+\bar{\mu}_i\left(x_{2i-1}-\sqrt{-1}x_{2i}\right)^2$ has one positive and one negative eigenvalue, since $\mu_i \neq 0$. Therefore the form $\bar{Q}_{\ell}$ has at least $m$ negative eigenvalues, and thus $\ql$ is strictly negative on a subspace of dimension at least $m$.

Recall that the form $\ql$ is positive semidefinite, which implies that $\bar{Q}_{\ell}$ is psd on the image of $E_{\rr}$. Thus the image of $E_{\rr}$ has codimension at least $m$ and the Lemma follows.
\end{proof}

We can restate the Lemma as follows: note that $|S|=k+2m$ and $|S|-c$ is the number of linearly independent relations that evaluation on  $S$ imposes on forms in $\hd$. Hence we get the following Corollary:

\begin{cor}
Suppose that $\ql$ is positive semidefinite, then the number of complex conjugate pairs in $S$ is at most equal to the number of linearly independent relations on $S$ for forms of degree $d$.
\end{cor}


Applying this to transversal intersections in the cases $(3,6)$ and $(4,4)$ we get:
\begin{cor}\label{NEWCOR Number of Zeroes}
Suppose that $\ell$ is an extreme ray of $\sq^*$ that does not correspond to point evaluation, and let $f_1,\ldots,f_{n-1}$ be forms in the kernel $\wl$ of $\ql$ intersecting transversely in $s=d^{n-1}$ points, $\Gamma=\{\gamma_1,\dots,\gamma_s\}$. Then the set $\Gamma$ includes at most $1$ complex conjugate pair and the rest of the points in $\Gamma$ are real.
\end{cor}

\section{Proofs of Main Theorems}\label{NEW SECTION Proofs}
We now prove Theorem \ref{THEOREM MAIN (3,6)} and Theorem \ref{THEOREM MAIN (4,4)} in a unified manner.
\begin{proof}[Proof of Theorems \ref{THEOREM MAIN (3,6)} and \ref{THEOREM MAIN (4,4)}]
Suppose that $p \in \p$ and $p$ is not sos. Then there exists an extreme ray $\ell$ of $\sq^*$ such that $\ell(p) <0$ and $\ell(q)\geq 0$ for all $q \in \sq$. Since $\ell(p) <0$ and $p$ is nonnegative it follows that $\ell$ does not correspond to point evaluation. Let $\ql$ be the quadratic from associated with $\ell$ and let $\wl$ be the kernel of $\ql$. Then by Theorem \ref{COROLLARY Dimension Kernel} we have $\dim \wl=n$ and by Lemma \ref{NEWLEMMA Bertini} we can find $f_1,\dots,f_{n-1} \in \wl$ intersecting transversely in $s=d^{n-1}$ projective points $\gamma_1,\ldots,\gamma_s$. By Corollary \ref{NEWCOR Number of Zeroes} we know that at most two of $\gamma_i$ are complex and by Lemma \ref{NEWLEMMA Relation Structure} the linear functional $\ell$ has the desired form.
\end{proof}


\begin{proof}[Proof of Corollaries \ref{NEWCOR Pos 3,6} and \ref{NEWCOR Pos 4,4}]
Let $p$ be a strictly positive form on the boundary of $\sq$. Then there exists an extreme ray $\ell$ of the dual cone $\sq^*$, such that $\ell(p)=0$. Now suppose that $p=\sum f_i^2$ for some $f_i \in \hd$. It follows that $\ql(f_i)=0$ for all $i$ and since $\ql$ is a positive semidefinite quadratic form we see that all $f_i$ lie in the kernel $\wl$ of $\ql$. By Theorem \ref{COROLLARY Dimension Kernel} we know that $\dim \wl=n$ and therefore $p$ is a sum of squares of forms coming from a $n$ dimensional subspace of $\h$. It follows that $p$ is a sum of at most $n$ squares.

Now suppose that $p$ is a sum of $n-1$ or fewer squares, $p=f_1^2+\dots+f_{n-1}^2$ with some $f_i$ possibly zero. Since $p$ is strictly positive we know that the forms $f_i$ have no common real zeroes. Therefore we found $n-1$ forms $f_i \in \wl$ that have no common real zeroes (if $p$ was a sum of fewer than $n-1$ squares then we can add arbitrary $f_i$ to get their number up to $n-1$). By the proof of Lemma \ref{NEWLABEL 4,4} we know that $n-1$ generic forms in $\wl$ intersect transversely, and hence we can find forms $f_i' \in \wl$ in a neighborhood of $f_i$  such that $f_i'$ intersect transversely in $d^{n-1}$ complex points. This is a contradiction by Corollary \ref{NEWCOR Number of Zeroes}.
\end{proof}

We now examine the two cases of Corollary \ref{NEWCOR Number of Zeroes}: The case of fully real intersection and the case of one complex conjugate pair. In each of these cases there exist psd forms $\ql$ corresponding to extreme rays of $\sq^*$. We provide explicit equations of these extreme rays, based on the Cayley-Bacharach relation, thus giving us a complete description of the extreme rays of $\sq^*$.

\section{Fully Real Intersection} \label{SECTION FULLY REAL INTERSECTION}

We have already examined the difference between attainable values on fully real intersection for psd and sos forms in Section \ref{SECTION GEOMETRY OF EVALUATIONS}. Now we describe the dual picture of all the linear inequalities that come from fully real intersections, which hold on $\sq$ but fail on $\p$.

Suppose that for the cases $(3,6)$ and $(4,4)$ a linear functional $\ell \in \h^*$ spans an extreme ray of $\Sigma_{n_2d}^*$ that does not correspond to point evaluation. Let $\wl$ be the kernel of $\ql$ and suppose that $f_1,\dots,f_{n-1} \in \wl$ intersect transversely in $s=d^{n-1}$ real projective points $\gamma_1,\ldots,\gamma_s$. Let $v_1,\ldots,v_s$ be affine representatives for $\gamma_1,\dots,\gamma_s$ and let $u_1p(v_1)+\ldots+u_sp(v_s)=0$ with $u_i \in \mathbb{R}$ be the unique Cayley-Bacharach relation on the points $v_i$.
\begin{theorem}\label{UNI THEOREM Real Intersection}
The form $\ql$ can be uniquely written as
$$\ql(f)=a_1f(v_1)^2+\ldots+a_sf(v_s)^2 \foreq f \in H_{n,d},$$ with a single negative coefficient $a_k$, the rest of the $a_i$ positive and
$$\sum_{i=1}^s \frac{u_i^2}{a_i}=0.$$
Furthermore any such form is extreme in $\sq^*$
\end{theorem}




The key to the unified description in these cases is the uniqueness of the Cayley-Bacharach relation, which holds for both $(3,6)$ and $(4,4)$.

Let $E: \hd \rightarrow \rr^s$ be the evaluation map that sends $f \in \hd$ to its values at the points $v_i$:
$$E(f)=(f(v_1),\ldots,f(v_s)).$$

Let $L$ be the image of $\hd$ under $E$. Since the forms in $\hd$ satisfy a unique relation, it follows that $L$ is the following hyperplane:
$$L=\left\{x \in \rr^s \hspace{2mm} \mid \hspace{2mm} u_1x_1+\ldots+u_sx_s=0 \right\}.$$

We would like to classify all positive semidefinite quadratic forms $\ql$ on $\hd$ with $$Q_{\ell}=a_1f^2(v_1)+\ldots+a_sf^2(v_s),$$ and coefficients $a_i \in \mathbb{R}$. By Lemma \ref{NEWLEMMA Relation Structure} the extreme rays $\ell$ of $\sq^*$ are guaranteed to have this form with points $v_i$ coming from transverse intersection of 2 cubics or 3 quadratics. In terms of the evaluations map we would like to find all quadratic forms $Q: \rr^s \rightarrow \mathbb{R}$ given by $Q=a_1x_1^2+\ldots+a_sx^2_s$ that are positive semidefinite on the hyperplane $L$.

Let $S_L$ be the cone of diagonal quadratic forms $Q=a_1x_1^2+\ldots+a_sx^2_s$ that are positive semidefinite on the hyperplane $L$. Theorem \ref{UNI THEOREM Real Intersection} follows immediately from the following proposition:
\begin{prop}\label{PROP FULLY REAL INTERSECTION}
Suppose that $Q$ spans an extreme ray of $S_L$. Then either $Q=a_ix_i^2$ for some $i$ and $a_i>0$ or $Q$ has the form specified in Theorem \ref{UNI THEOREM Real Intersection}.
\end{prop}
\begin{proof}
Let $Q=a_1x_1^2+\ldots+a_sx^2_s$ span an extreme ray of $S_L$. If all coefficients $a_i$ are nonnegative then since $Q$ spans an extreme ray it follows that $Q=a_ix_i^2$ for some $i$ and $a_i>0$.

Suppose now that one of the coefficients $a_i$ is zero. Without loss of generality we may assume $a_s=0$. Then we claim that $Q=a_ix_i^2$ for some $i<s$ and $a_i>0$.

First we show that all other coefficients must be nonnegative. Suppose that $a_s=0$ and $a_1 <0$. From the equation of $L$ we can write $x_1=-(u_2x_2+\ldots+u_sx_s)/u_1$. Therefore the form
$$Q=a_1\frac{(u_2x_2+\ldots+u_sx_s)^2}{u_1^2}+a_2x_2^2+\ldots+a_{s-1}x_{s-1}^2$$
is positive semidefinite for all values of $x_2,\ldots,x_s$. However, the coefficient of $x_s^2$ is strictly negative, which is a contradiction. Therefore we can write $Q=a_1x_1^2+\ldots+a_{s-1}x_{s-1}^2$ with $a_i \geq 0$. Since $Q$ spans an extreme ray it follows that $Q=a_ix_i^2$ for some $i<s$ and $a_i>0$.

Next we claim that if one of $a_i$ is negative then the rest are strictly positive. Suppose that $a_1<0$ and $a_2 \leq 0$. Then again write $x_1=-(u_2x_2+\ldots+u_sx_s)/u_1$ and
$Q=a_1\frac{(u_2x_2+\ldots+u_sx_s)^2}{u_1^2}+a_2x_2^2+\ldots+a_{s}x_{s}^2$. Now the coefficient of $x_2^2$ is strictly negative, which is a contradiction.

Now we have only one case left: one $a_i$ is negative and the rest are strictly positive. Suppose that $a_s <0$. Write $x_s=-(u_1x_1+\ldots+u_{s-1}x_{s-1})/u_s$
and $$Q=a_1x_1^2+\ldots+a_{s-1}x_{s-1}^2+a_s \frac{(u_1x_1+\ldots+u_{s-1}x_{s-1})^2}{u_s^2}.$$

Let's maximize $\frac{(u_1x_1+\ldots+u_{s-1}x_{s-1})^2}{u_s^2}$ subject to $a_1x_1^2+\ldots+a_{s-1}x_{s-1}^2=1$. Applying Lagrange multipliers we see that
$x_i=\lambda u_i/a_i$ for some $\lambda$ and all $i \leq s-1$. Now we find the value of $a_s$ that makes $Q(u_1/a_1,\ldots,u_{s-1}/a_{s-1})=0$. We see that this happens for $$a_s^{\star}=\frac{-u_s^2}{\frac{u_1^2}{a_1}+\ldots+\frac{u_{s-1}^2}{a_{s-1}}}.$$ It is clear that any $a_s \geq a_s^{\star}$ will result in a psd form $Q$. However, if $a_s > a_s^{\star}$ then the form $Q$ is positive definite on $L$ and therefore it does not lie on the boundary of $S_L$ and does not span an extreme ray.

With $a_s=a_s^{\star}$ the kernel of $Q$  is spanned by the vector $v=\left(u_1/a_1,\ldots,u_s/a_s\right)$. We see that (up to a constant multiple) $Q$ is the only form in $S_L$ with kernel that includes $v$. Therefore $Q$ is extreme in $S_L$.
\end{proof}
\section{One Complex Pair}\label{SECTION ONE COMPLEX PAIR}
We now examine the last case of intersection with one complex conjugate pair of zeroes. Suppose that $\ell$ spans an extreme ray of $\sq^*$ that does not correspond to point evaluation. Let $\wl$ be the kernel of $\ql$ and suppose that $f_1,\dots,f_{n-1} \in \wl$ intersect transversely in $s=d^{n-1}$ projective points $\gamma_1,\ldots,\gamma_s$ with a single complex conjugate pair and the rest of $\gamma_i$ real. Let $v_1,\ldots,v_{s-2}$ be affine representatives for the real $\gamma_i$ and $z, \bar{z}$ be affine representatives for the complex roots chosen such that $$u_1f(v_1)+\ldots+u_{s-2}f(v_{s-2})+f(z)+f(\bar{z})=0,$$ \noindent with $u_i \in \mathbb{R}$, is the unique Cayley-Bacharach relation on the points $v_i,z,\bar{z}$.
\begin{theorem}\label{UNI THEOREM Complex Pair}
The form $\ql$ can be uniquely written as
$$\ql(f)=a_1f(v_1)^2+\ldots+a_{s-2}f(v_{s-2})^2+4m\left(\operatorname{Re} z\right)^2-4t\left(\operatorname{Im} z\right)^2 \foreq f \in \hd, $$ with all $a_i>0$ and $m$ and $t$ satisfying

$$\frac{2m}{m^2+t^2}+\sum_{i=1}^{s-2}\frac{u_i^2}{a_i}=0.$$
Furthermore any such form is extreme in $\sq^*$.
\end{theorem}




Again we give a unified presentation based on the uniqueness of the Cayley-Bacharach relation. 
We construct the real evaluation map $E_{\rr}: \hd \rightarrow \mathbb{R}^{s}$ as follows:
$$E_{\rr}(f)=(f(v_1),\ldots,f(v_{s-2}),2\operatorname{Re} f(z), 2\operatorname{Im} f(z)).$$
\noindent Let $L$ be the image of $\hd$ under $E$. Then $L$ is the following hyperplane:
$$L=\left\{x \in \mathbb{R}^{s} \hspace{2mm} \mid \hspace{2mm} u_1x_1+\ldots+u_{s-2}x_{s-2}+x_{s-1}=0\right\}.$$
Note that $L$ does not depend on $x_{s}=2\operatorname{Im} f(z).$

We would like to classify all positive semidefinite quadratic forms $\ql$ on $\hd$ with $$Q_{\ell}=a_1f^2(v_1)+\ldots+a_{s-2}f^2(v_{s-2})+bf^2(z)+\bar{b}f^2(\bar{z}),$$ and coefficients $a_i \in \mathbb{R}$ and $b \in \mathbb{C}$. By Lemma \ref{NEWLEMMA Relation Structure} the extreme rays $\ell$ of $\sq^*$ are guaranteed to have this form with points $v_i$ and $z$ coming from transverse intersection of 2 cubics or 3 quadratics. 

Let $b=m+t\sqrt{-1}$. In terms of the evaluation map we would like to find all quadratic forms $Q: \mathbb{R}^{s} \rightarrow \mathbb{R}$ given by
$$Q=a_1x_1^2+\ldots+a_{s-2}x^2_{s-2}+\frac{m}{2}\left(x_{s-1}^2-x_{s}^2\right)-tx_{s-1}{x_{s}}.$$
that are positive semidefinite on the hyperplane $L$. Let $S_L$ be the convex cone of all such quadratic forms.


\begin{prop}
Suppose that $Q$ spans an extreme ray of $S_L$. Then either $Q=a_ix_i^2$ for some $i$ and $a_i>0$ or $Q$ has the form specified by Theorem \ref{UNI THEOREM Complex Pair}.
Conversely, all such forms span extreme rays of $S_L$.
\end{prop}
\begin{proof}
Let $Q=a_1x_1^2+\ldots+a_{s-2}x^2_{s-2}+\frac{m}{2}\left(x_{s-1}^2-x_{s}^2\right)-tx_{s-1}{x_{s}}$ span an extreme ray of $S_L$.
If $m=0$ then in order for $Q$ to be psd on $L$ we must have $t=0$ and then all the coefficients $a_i$ are nonnegative and $Q$ is a nonnegative combination of point evaluations. Since $Q$ is extreme in $S_L$ it follows that $Q=a_ix_i^2$ for some $i$ and $a_i>0$.

If $m \neq 0$ then it must be strictly negative since $x_{s}^2$ is not constrained by $L$ and its coefficient is $-m/2$. We can complete the square in $Q$ and write
$$Q=a_1x_1^2+\ldots+a_{s-2}x^2_{s-2}+\frac{m^2+t^2}{2m}x_{s-1}^2-\frac{1}{2m}(tx_{s-1}+mx_{s})^2.$$

Since the term $-\frac{1}{2m}(tx_{s-1}+mx_{s})^2$ is always nonnegative and $x_{s}$ is unconstrained, we can always make it equal to zero by taking $x_{s}=-tx_{s-1}/m$. Therefore  $Q$ is psd if and only if $Q'=a_1x_1^2+\ldots+a_{s-2}x^2_{s-2}+\frac{m^2+t^2}{m}x_{s-1}^2$ is psd on $L'=\{x \in \mathbb{R}^{s-1} \hspace{2mm} \mid \hspace{2mm} u_1x_1+\ldots+u_{s-2}x_{s-2}+x_{s-1}=0\}$. We are in exactly the same situation as the case of fully real intersection and since the coefficient of $x_{s-1}$ is guaranteed to be negative we know from Proposition \ref{PROP FULLY REAL INTERSECTION} that all $a_i$ are positive and $$\frac{m^2+t^2}{2m}=\frac{-1}{\frac{u_1^2}{a_1}+\ldots+\frac{u_{s-2}^2}{a_{s-2}}}.$$

The resulting quadratic form will have a unique projective zero in $L$ at $$v=\left(\frac{u_1}{a_1},\ldots,\frac{u_{s-2}}{a_{s-2}},\frac{2m}{m^2+t^2},\frac{-2t}{m^2+t^2}\right).$$

It is easy to verify that up to a constant multiple there is a unique form $Q$ in $S_L$ with $v$ in the kernel, which guarantees that $Q$ is extreme and completes the proof.
\end{proof}

We would like to close the paper with a conjecture. Let $\wl$ be the kernel of an extreme ray $\ql$ of $\sq^*$, in the cases $(3,6)$ and $(4,4)$, that does not correspond to point evaluation. It follows from Corollary \ref{NEWCOR Number of Zeroes} that \textit{any} transversal intersection of $n-1$ forms in the kernel $\wl$ has at most one complex pair of zeroes. We have examined extreme rays that can be defined on such transverse intersections in Sections \ref{SECTION FULLY REAL INTERSECTION} and \ref{SECTION ONE COMPLEX PAIR}. However, we conjecture that the case of one complex pair of zeroes is not truly necessary, as we can always find a fully real intersection inside $\wl$:

\begin{conj}
Suppose that for the cases $(3,6)$ and $(4,4)$, $W$ is an $n$-dimensional subspace of $\h$ such that $\mathcal{V}_{\co}(W)=\emptyset$ and any collection of forms $f_1,\dots,f_{n-1}$ intersecting transversely in $W$ has at most $1$ complex pair of zeroes. Then there exist forms $f_1,\ldots,f_{n-1} \in W$ intersecting transversely in only real points.
\end{conj}

\section*{acknowledgments}
The author would like to thank anonymous referees whose suggestions greatly helped in improving the exposition and clarity of the paper, and the hospitality of the Institute for Pure and Applied Math at UCLA, where the paper was written.

\bigskip
\medskip
\noindent
Grigoriy Blekherman, Georgia Inst.~of Technology, Atlanta, USA,
{\tt greg@math.gatech.edu}

\end{document}